\documentclass[a4paper,10pt]{amsart}
\usepackage{amsmath,amssymb,graphicx,url}

\newtheorem{theorem}{Theorem}[section]
\newtheorem{prop}[theorem]{Proposition}
\newtheorem{coro}[theorem]{Corollary}
\newtheorem{lemma}[theorem]{Lemma}

\newtheorem*{maintheorem1}{Proposition \ref{maintheorem}}
\newtheorem*{maintheorem2}{Theorem \ref{losurgeryinlensspace}}
\newtheorem*{corollary4.6}{Corollary \ref{hyperboliccase}}

\theoremstyle{definition}
\newtheorem{defin}[theorem]{Definition}

\newtheorem{rem}[theorem]{Remark}

\begin{document}
\title[]{Integral left-orderable surgeries on genus one fibered knots}
\author{Kazuhiro ICHIHARA \and Yasuharu NAKAE}
\address{Department of Mathematics, College of Humanities and Sciences, Nihon University, 3-25-40 Sakurajosui, Setagaya-ku, Tokyo 156-8550, Japan.}
\email{ichihara.kazuhiro@nihon-u.ac.jp}
\address{Graduate School of Engineering Science, Akita University,
1-1 Tegata Gakuen-machi, Akita city, Akita, 010-8502, Japan.}
\email{nakae@math.akita-u.ac.jp}

\begin{abstract}
Following the classification of genus one fibered knots in lens spaces by Baker,
we determine hyperbolic genus one fibered knots in lens spaces
on whose all integral Dehn surgeries yield closed 3-manifolds with left-orderable fundamental groups.
\end{abstract}

\maketitle

\section{Introduction}
In this paper,
we study the left-orderability of the fundamental group of a closed manifold
which is obtained by Dehn surgery on a genus one fibered knot in a lens space.
In \cite{BGW},
Boyer, Gordon, and Watson formulated a conjecture
which states that an irreducible rational homology sphere is an L-space
if and only if its fundamental group is not left-orderable.
Here, a rational homology 3-sphere $Y$ is called an \textit{L-space}
if $\mathrm{rank}\;\widehat{HF}(Y)=|H_1(Y;\mathbb{Z})|$ holds,
and a non-trivial group $G$ is called \textit{left-orderable}
if there is a total ordering on the elements of $G$
which is invariant under left multiplication.
This famous conjecture proposes a topological characterization of an L-space
without referring to Heegaard Floer homology.
Thus it is interesting to determine which 3-manifold has a left-orderable fundamental group.

One of the methods yielding many L-spaces is Dehn surgery on a knot.
It is the operation to create a new closed 3-manifold
which is done by removing an open tubular neighborhood of the knot
and gluing back a solid torus via a boundary homeomorphism.
Thus it is natural to ask
when 3-manifolds obtained by Dehn surgeries have non-left-orderable fundamental groups.
In \cite{RS},
Roberts and Shareshian give a condition that
the manifold obtained by Dehn surgery on a genus one fibered knot
has the non-left-orderable fundamental group.
Here, a knot in a closed 3-manifold is called a \textit{genus one fibered knot}
if its exterior is a punctured torus bundle over the circle
and the knot is parallel to the boundary of a fiber.
In contrast with this work,
we study manifolds obtained by Dehn surgery on a genus one fibered knot
with the left-orderable fundamental group.
For a genus one fibered knot with the orientation preserving hyperbolic monodromy,
we can obtain the following by applying a theorem of Fenley~\cite{Fen94}.

\begin{maintheorem1}
Let $M$ be a closed, orientable 3-manifold and
$K$ be a genus one fibered knot in $M$.
If the monodromy matrix $\phi_\sharp$ of $K$ satisfies
$\mathrm{Trace} (\phi_\sharp)>2$
then the fundamental group $\pi_1 (\Sigma_K(n))$ is left-orderable for any integer $n$,
where $\Sigma_K(n)$ is the closed manifold obtained by Dehn surgery on $K$ with slope $n$. \end{maintheorem1}

This fact is originally introduced in \cite[p.1217]{BRW} as a comment of Roberts
for the figure-eight knot in $S^3$,
and also mentioned in \cite[p.709]{RS}.
We will explain the background of the fact and formulate a proof of the fact in Section 3.

Some equivalent classes of genus one fibered knots in lens spaces are studied by Morimoto~\cite{Mo89},
and all such classes are completely classified by Baker~\cite{Ba2014}.
Following the classification,
by applying Proposition~\ref{maintheorem},
we determine the genus one fibered knots in lens spaces
on whose all integral Dehn surgeries
yield closed 3-manifolds
with left-orderable fundamental groups.

\begin{maintheorem2}
Let $K$ be a genus one fibered knot in a lens space $L(\alpha, \beta)$.
If the monodromy matrix $\phi_\sharp$ of $K$ is conjugate to one of the following types
{\rm (1)} and {\rm (2)}
in $GL_2(\mathbb{Z})$,
then for any $n\in\mathbb{Z}$,
the fundamental group $\pi_1(\Sigma_K(n))$ of the closed manifold $\Sigma_K(n)$ obtained by $n$-surgery
on $K$
is left-orderable.
\begin{enumerate}
%\item $L(\alpha,\beta)=L(4,1)$ and $\phi_\sharp=\begin{pmatrix} 5 & 4 \\ 1 & 1 \end{pmatrix}$,
\item $L(\alpha,\beta)=L(\alpha,1)$ and $\phi_\sharp=\begin{pmatrix} 1+\alpha & \alpha \\ 1 & 1 \end{pmatrix}$ for $\alpha>0$, or
\item $L(\alpha,\beta)$ satisfies $\alpha=2pq+p+q+1$, $\beta=2q+1$ and \\ $\phi_\sharp=\begin{pmatrix} 2pq+p-q & 2pq+3p-q-1 \\ 2q+1 & 2q+3 \end{pmatrix}$ for integers $p,q>0$.
\end{enumerate}
\end{maintheorem2}
 
As a corollary, we have the following by considering the cases for hyperbolic genus one fibered knots in lens spaces.

\begin{corollary4.6}
Let $K$ be a hyperbolic, genus one fibered knot in lens space $L(\alpha, \beta)$.
The fundamental groups of the closed manifolds obtained by all integral surgeries on K are left-orderable
if and only if the monodromy matrix $\phi_\sharp$ of $K$
is conjugate to one of the types {\rm (1)} and {\rm (2)}
of Theorem~\ref{losurgeryinlensspace} in $GL_2(\mathbb{Z})$.
\end{corollary4.6}

There are many works for the left-orderability of the fundamental group of 3-manifolds which are obtained by
Dehn surgeries on knots in $S^3$.
For instance,
let $K$ be the figure-eight knot in $S^3$,
then
$\pi_1(\Sigma_K(r))$ is left-orderable for $-4\leq r \leq 4$~\cite[for $-4<r<4$]{BGW}, \cite[for $r=\pm 4$]{CLW}.
These results are achieved by using representations from the fundamental group to $SL_2(\mathbb{R})$.
In this direction,
there are some works,
by Hakamata and Teragaito for hyperbolic genus one 2-bridge knots~\cite{HT2014},
Clay and Teragaito for 2-bridge knots~\cite{CT2013},
and 
Teragaito~\cite{Tera2013} and Tran~\cite{Tran2015} for twist knots.
By using a criterion of Culler and Dunfield~\cite{CD2018},
Nie showed that for $(-2,3,2s+1)$-pretzel knot $K$ ($s\ge 3$),
there exists $\varepsilon>0$ such that $\pi_1(\Sigma_K(r))$ is left-orderable for any rational number
$r\in (-\varepsilon, \varepsilon)$~\cite{Nie2019},
and Tran showed that for some twisted torus knots $K$,
$\pi_1(\Sigma_K(r))$ is left-orderable for any $r\in\mathbb{Q}$ sufficiently close to $0$~\cite{Tran2019}.

In contrast with these methods to prove the left-orderability,
we use the existence of $\mathbb{R}$-covered Anosov flow
as follows.
A countable group $G$
is left-orderable if there exists a faithful action of $G$ on the real line $\mathbb{R}$ (see Section 2).
If a closed 3-manifold $M$ contains a transversely orientable $\mathbb{R}$-covered foliation $\mathcal{F}$,
the fundamental group $\pi_1(M)$ acts faithfully to the leaf space which is homeomorphic to $\mathbb{R}$,
so $\pi_1(M)$ is left-orderable.
A closed 3-manifold equipped with a suspension Anosov flow
is naturally endowed with $\mathbb{R}$-covered foliations
which are the stable and unstable foliations of the flow.
An Anosov flow is said to be $\mathbb{R}$-covered
if both the stable and unstable foliations of the flow are $\mathbb{R}$-covered.
In \cite{Fen94}, Fenley showed that
closed manifolds obtained by Dehn surgeries on a closed orbit of a suspension Anosov flow
with integral slopes
contain $\mathbb{R}$-covered Anosov flows.
Then we can see the fundamental group of the resultant closed 3-manifold is left-orderable.
We think this method will open up a new study for the existence of left-orderable fundamental groups.

\begin{rem}
The manifolds admitting a genus one open book decomposition
with connected binding
are obtained by certain Dehn surgeries on genus one fibered knots.
In \cite{Bal},
for manifolds admitting a genus one open book decomposition
with connected binding,
Baldwin completely determines which these manifolds are L-space,
and in \cite{LW},
Li and Watson show that the fundamental groups of these manifolds which are L-spaces
are not left-orderable.
In \cite{BH},
Boyer and Hu show that the L-space conjecture is true for irreducible 3-manifolds
which admit a genus one open book decomposition with connected binding
by combining Theorem 1.9 in \cite{BH} with these results of Baldwin, and Li and Watson.
\end{rem}

This paper is organized as follows.
In Section 2, we review definitions and some results about $\mathbb{R}$-covered foliations, left-orderable groups, 
once punctured torus bundles, and genus one fibered knots.
In Section 3, we briefly review results about Anosov flow and Dehn surgery on its closed orbits.
After mentioning a theorem of Fenley \cite{Fen94}, we prove Proposition~\ref{maintheorem}.
In Section 4,
we describe first a theorem of Baker \cite{Ba2014},
then we prove Theorem~\ref{losurgeryinlensspace}.

\section{Preliminaries}

We shall review some definitions and results about foliations, left-orderable groups, and once punctured torus bundles in this section.

\subsection{$\mathbb{R}$-covered foliations and left-orderable groups}

For definitions and fundamental results on foliation theory, see \cite{CC1}.

A codimension one foliation on a 3-manifold is \textit{transversely oriented} (or \textit{co-oriented})
if all leaves are coherently oriented.
A solid torus
is called a \textit{Reeb component}
if it is equipped with a codimension one foliation whose
leaves are all homeomorphic to a plane except for the torus boundary leaf.
A codimension one, transversely oriented foliation $\mathcal{F}$ on a closed 3-manifold
called a \textit{Reebless foliation} if $\mathcal{F}$ does not contain Reeb components as a leaf-saturated set.
By the theorems of Novikov\,\cite{No}, Rosenberg\,\cite{Ro}, and Palmeira\,\cite{Pa},
if a closed 3-manifold $M$ is not homeomorphic to $S^2\times S^1$ and contains a Reebless foliation,
then 
the fundamental group $\pi_1(M)$ is infinite,
the universal cover $\widetilde{M}$ is homeomorphic to $\mathbb{R}^3$,
and all leaves of its lifted foliation $\widetilde{\mathcal{F}}$ on $\widetilde{M}$ are homeomorphic to a plane
(see also \cite[Chapter 9]{CC2}).

If $\mathcal{F}$ is a Reebless foliation on a closed 3-manifold $M$,
the quotient space $\mathcal{T}=\widetilde{M}/\widetilde{\mathcal{F}}$
is called the \textit{leaf space} of $\mathcal{F}$.
In general $\mathcal{T}$ is a simply connected 1-dimensional manifold,
but it might be non-Hausdorff space \cite{HR} (see also \cite[\textit{Appendix D}]{CC2}).
Since the fundamental group $\pi_1(M)$ acts on the universal cover $\widetilde{M}$ as deck transformations,
the action descends to a nontrivial action of $\pi_1(M)$ on the leaf space $\mathcal{T}$
as homeomorphisms.
We call a Reebless foliation $\mathcal{F}$ is \textit{$\mathbb{R}$-covered}
if the leaf space $\mathcal{T}$ of $\mathcal{F}$ is homeomorphic to $\mathbb{R}$.
When an $\mathbb{R}$-covered foliation $\mathcal{F}$ of $M$ is transversely oriented,
the fundamental group $\pi_1(M)$ acts on the leaf space $\mathcal{T}$ as orientation preserving homeomorphisms.
Therefore we can obtain the following.

\begin{prop}\label{rcovfoliaction}
Let $\mathcal{F}$ be a transversely oriented Reebless foliation
on a closed orientable 3-manifold $M$.
If $\mathcal{F}$ is $\mathbb{R}$-covered,
there is a faithful action of $\pi_1(M)$ on the real line $\mathbb{R}$
as orientation preserving homeomorphisms.
\end{prop}

A group $G$ is \textit{left-orderable}
if there is a total ordering $<$ on every elements of $G$ which is invariant under left multiplication,
that is, for any $f,g,h \in G$, $g<h$ implies $fg<fh$.
It is known that
a non-trivial countable group $G$ is left-orderable
if and only if
there exists a faithful action of $G$ on the real line $\mathbb{R}$
(cf. \cite[Proposition 5.1]{BC}).
Then we can obtain the following by combining this observation with Proposition~\ref{rcovfoliaction}.

\begin{prop}\label{rcovleftorderable}
If a closed orientable 3-manifold $M$ admits a transversely oriented $\mathbb{R}$-covered foliation,
then its fundamental group $\pi_1(M)$ is left-orderable.
\end{prop}

\subsection{Genus one fibered knots}

We will briefly review the properties of once punctured torus bundles following the description of \cite{RS}.
Let $T$ be a compact, genus one surface with one boundary component,
and $\phi:T\to T$ be a homeomorphism on $T$ fixing the boundary $\partial T$ pointwise.
The punctured torus bundle $M_\phi$ with monodromy $\phi$ is defined by the following:
$$M_\phi=(T\times [0,1])/(x,1)\sim(\phi(x),0).$$

We can identify the induced automorphism $\phi_\sharp:H_1(T)\to H_1(T)$
with a $2\times2$ non-singular matrix
since $H_1(T)\cong \mathbb{Z}\oplus\mathbb{Z}$.
If $\phi$ is an orientation preserving homeomorphism,
we can see $\phi_\sharp \in SL_2(\mathbb{Z})$
and $M_\phi$ is orientable.
Two punctured torus bundles
$M_\phi$ and $M_\psi$ are homeomorphic
if and only if
$\phi_\sharp$ and $\psi_\sharp$ are in the same conjugacy class of $GL_2(\mathbb{Z})$
(see \cite[Proposition 1.3.1]{CJR}).

We fix a meridian-longitude pair on $\partial M_\phi$ as follows
(for general definitions and results of Dehn filling and Dehn surgery, see \cite{Rol}).
We define the longitude $\ell$ as the simple closed curve $\partial T\times \{0\}$ on $\partial M_\phi$ .
For a point $x$ in $\partial T\times \{0\}$ the simple closed curve $m=\{x\}\times [0,1]$ intersects $\ell$ at only $x$,
then we can recognize $m$ as a meridian.
The isotopy class $\gamma$ of each non-trivial unoriented simple closed curve in the torus $\partial M_\phi$ 
is called \textit{slope}, and a representative $c$ in the isotopy class is called a simple closed curve of slope $\gamma$. 
Using the meridian-longitude pair,
for any slope $\gamma$ on $\partial M_\phi$,
an oriented simple closed curve $c$ of slope $\gamma$
is represented by $[c]=p[m]+q[l]\in H_1(\partial M_\phi)$, $p,q\in\mathbb{Z}$ and $p$ and $q$ are relatively prime.
So we identify the slope $\gamma$ by the rational number $p/q$. 

In the following way,
a closed 3-manifold is constructed from $M_\phi$,
which we denote by
$M_\phi(p/q)$.
Let $V=D^2\times S^1$ be the solid torus and $f:\partial V\to \partial M_\phi$ be a homeomorphism
which satisfies $f(\partial D^2\times \{t\})=c$ for a fixed $t\in S^1$ and a simple closed curve $c$
on $\partial M_\phi$ of slope $p/q$.
Then we define
$$M_\phi(p/q)=(M_\phi\cup V)/(x\sim f(x)).$$
In fact,
the homeomorphism type of $M_\phi(p/q)$ does not depend on the choice of $t$ or $f$,
only depends on the monodromy $\phi$ and the slope $p/q$.

We denote $M_\phi(1/0)$ by $\bar{M}_\phi$ hereafter.
A knot $K$ in a closed 3-manifold $M$ is called a \textit{genus one fibered knot} (for short, a \textit{GOF-knot})
if $M\setminus \mathrm{int} N(K)$ is homeomorphic to $M_\phi$ for some orientation preserving homeomorphism $\phi$,
$K$ is parallel to the boundary $\partial T$ of a fiber $T$,
and the monodromy $\phi$ is the identity on the boundary of the fiber. 
We can regard 
a closed 3-manifold including a GOF-knot $K$ as $\bar{M}_\phi$ for some $\phi$,
and $K$ is the core curve in $V$.

The closed 3-manifold $\bar{M}_\phi$ is identified with the double branched cover of $S^3$
over a closed 3-braid as follows (see also \cite[Section 4 and 5]{MR}, \cite[Section 2]{Bal}).

We take the once punctured torus $T$ as shown in Figure~\ref{oncepuncturedtorus}.
\begin{figure}[h]
	\centerline{\includegraphics[keepaspectratio]{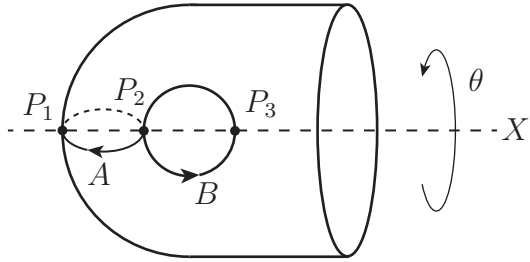}}
\caption{once punctured torus $T$ with curves $A$ and $B$}
\label{oncepuncturedtorus}
\end{figure}
Let $\theta : T\to T$ be the involution of $T$ obtained by the $180$ degrees rotation about the axis $X$,
where $P_1$, $P_2$ and $P_3$ are the intersection points of $T$ and $X$.
Then we obtain the projection map $p:T\to T/\theta= D^2$ by taking the quotient of $\theta$
which is the double branched cover of the $2$-disk $D^2$
with the three branched points $P_1$, $P_2$ and $P_3$,
here we denote the image $p(P_i)$ by the same symbol $P_i$ on $D^2$. 
Let $\phi$ be a homeomorphism $\phi:T\to T$ which is the identity on the boundary $\partial T$.
Since $\phi$ is isotopic to a composition of Dehn twists along the curves $A$ and $B$
in Figure~\ref{oncepuncturedtorus},
we can suppose that $\phi$ maps $\{P_1, P_2, P_3\}$ into itself.
The involution $\theta$ commutes $\phi$,
so $\theta$ extends on each fiber of $M_\phi$.
Therefore we obtain the double branched covering $\bar{p}$ of the solid torus $\bar{V}$
$$\bar{p}: M_\phi=T\times [0,1]/\phi \to D^2\times [0,1]/\phi_\theta =\bar{V}$$
with the extended monodromy $\phi_\theta$ on $D^2$.
The branch set is $L=\{P_1, P_2, P_3\}\times [0,1] \subset \bar{V}$. 
The set $L$ becomes a closed $3$-braid in the solid torus $\bar{V}$
because $\phi_\theta$ fixes the boundary $\partial D^2$ and
alternates $\{ P_1, P_2, P_3 \}$ each other.
We also extend the branched covering $\bar{p}$ to $\bar{M}_\phi$,
$$\hat{p}: \bar{M}_\phi=M_\phi\cup V \to \bar{V} \cup \bar{V}' \simeq S^3$$
where the attached solid torus $V$ is the double cover of the solid torus $\bar{V}'$.
Since $\bar{V}\cup\bar{V}'$ is homeomorphic to $S^3$,
we can obtain the double branched cover $\hat{p}$ of $S^3$
with the branched set $L$.

Let $\alpha$ and $\beta$ be the mapping classes
generated by Dehn twists along the curves $A$ and $B$ in Figure~\ref{oncepuncturedtorus}
respectively.
Each of classes corresponds to the matrix
of the induced automorphism on $H_1(T)$
in $SL_2(\mathbb{Z})$ as follows:
$$
\alpha \longleftrightarrow
\phi_A=\left(\begin{array}{cc}
1 & 1 \\
0 & 1
\end{array}\right)
\hspace{15pt}
\beta \longleftrightarrow
\phi_B=\left(\begin{array}{cc}
1 & 0 \\
1 & 1
\end{array}\right).
$$
And also each of classes corresponds to the element of the $3$-braid group
$B_3=\langle \sigma_1, \sigma_2 \rangle$ as follows:
$$\alpha \longleftrightarrow \sigma_1
\hspace{15pt}
\beta \longleftrightarrow \sigma_2^{-1},
$$
where $\sigma_1$ interchanges $P_1$ with $P_2$ counterclockwise on $D^2$
and also $\sigma_2^{-1}$ interchanges $P_2$ with $P_3$ clockwise from the upper view point.

\begin{figure}[h]
	\centerline{\includegraphics[keepaspectratio]{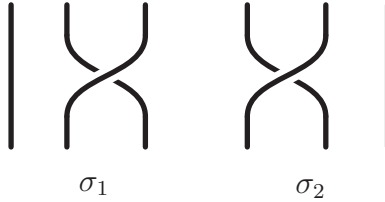}}
\caption{the generators $\sigma_1$ and $\sigma_2$ of the 3-braid group $B_3$}
\label{braid}
\end{figure}

The monodromy matrix $\phi_\sharp$ of a monodromy $\phi:T\to T$ is the composition of the sequence of $\phi_A$ or $\phi_B$ if
$\phi$ is an orientation preserving homeomorphism.
Therefore we can associate $\phi_\sharp$ to the element of $3$-braid group $B_3$
by the above observation as follows.

\begin{prop}\label{monodromybraid}
Let $K$ be a genus one fibered knot in a closed 3-manifold $M$ with the monodromy $\phi$.
$M$ is the double branched cover of $S^3$ over the closed $3$-braid $L$,
and we can regard $K$ with the lift of the braid axis of $L$
which is represented by the 3-braid $\sigma\in B_3$.
We suppose $\sigma$ is the product of generators $\sigma_1$ and $\sigma_2$ of $B_3$ as follows:
$$\sigma=\omega_1^{\varepsilon_1}\omega_2^{\varepsilon_2}\cdots\omega_n^{\varepsilon_n},\;\;\omega_i\in\{\sigma_1, \sigma_2\},\;\varepsilon_i\in\mathbb{Z}\setminus\{0\}.$$
Then the monodromy matrix $\phi_\sharp$ is obtained by the product of matrices $\phi_A$ and $\phi_B$
as follows:
$$\phi_\sharp=\phi_1^{\varepsilon_1}\phi_2^{\varepsilon_2}\cdots\phi_n^{\varepsilon_n},\;\;
\phi_i=
\left\{
\begin{array}{ll}
\phi_A & \text{if}\;\; \omega_i=\sigma_1 \\[3pt]
\phi_B^{-1} & \text{if}\;\; \omega_i=\sigma_2
\end{array}
\right.$$
\end{prop}

\section{Integral surgeries along a closed orbit of Anosov flow}

In this section,
we will review some definitions and results related to Anosov flow,
then explain Proposition~\ref{maintheorem} induced from a theorem of Fenley.

Anosov flow in a closed 3-manifold is a nonsingular flow with two directions
that are transverse to the flow direction,
one is exponentially expanding and the other is exponentially contracting. 
The precise definition is as follows.

\begin{defin}\label{DefofAnosovFlow} 
Let $M$ be a closed manifold and $\Phi_t$ a nonsingular smooth flow in $M$.
The flow $\Phi_t$ is called an \textit{Anosov flow}
if there exists a continuous $D\Phi_t$ invariant splitting of the tangent bundle $TM=TX\oplus E^s \oplus E^u$
such that there are constants $C \ge 1$ and $0<\lambda<1$ which satisfy the following:
\begin{enumerate}
\item $TX$ is one dimensional and tangent to $\Phi_t$,
\item $||D\Phi_t (v)|| \le C \lambda^t ||v||$ for any $v\in E^s$ and $t>0$,
\item $||D\Phi_{-t} (v) || \le C \lambda^t ||v||$ for any $v\in E^u$ and $t>0$.
\end{enumerate}
\end{defin}

An Anosov flow $\Phi_t$ defines
the plane fields $TX\oplus E^s$ and $TX\oplus E^u$,
and they are uniquely integrable.
Then these plane fields produce two foliations $\mathcal{F}^s$ and $\mathcal{F}^u$ respectively,
which are called (\textit{weak}) \textit{stable} and (\textit{weak}) \textit{unstable foliations} (see \cite{An69Tr}).
An Anosov flow $\Phi_t$ is said to be \textit{orientable}
if both $\mathcal{F}^s$ and $\mathcal{F}^u$ are transversely orientable,
\textit{$\mathbb{R}$-covered}
if both $\mathcal{F}^s$ and $\mathcal{F}^u$ are $\mathbb{R}$-covered.
Note that these foliations are Reebless
since the leaves of $\mathcal{F}^s$ and $\mathcal{F}^u$ are either
planes, open annuli or open M\"{o}bius bands.

There are two typical examples of Anosov flow, one is a geodesic flow on a unit tangent bundle
over a surface of negative constant curvature \cite{An69Tr}, the other is a suspension flow of Anosov diffeomorphism.
In this paper we mainly focus on Anosov flows generated by a suspension flow,
so we will review it shortly. 

A diffeomorphism $\phi:S\to S$ of a torus $S$ is called \textit{Anosov diffeomorphism}
if there exists a continuous splitting of the tangent bundle $TS=E^s\oplus E^u$
such that $E^s$ and $E^u$ are invariant by the derivative $D\phi$
and satisfy the conditions (2) and (3) of Definition \ref{DefofAnosovFlow}
for some constants $C$ and $\lambda$.
The suspension of a diffeomorphism $\phi$ is the flow $\Phi_t^\phi:\widehat{M}_\phi\to \widehat{M}_\phi$
on the mapping torus $\widehat{M}_\phi=S\times [0,1]/(x,1)\sim(\phi(x),0)$
given locally by $\Phi_t^\phi(x,s)=(x,s+t)$.
If the diffeomorphism $\phi$ is Anosov then the suspension $\Phi_t^\phi$ is an Anosov flow.

Several techniques are producing a new Anosov flow,
some of these are Fried's blowing-up technique ~\cite{Fri83},
Handel-Thurston's surgery construction ~\cite{HT80} and Goodman's surgery ~\cite{Goo83}.
We will review Fried's blowing-up technique as follows.

Let $\Phi_t$ be an Anosov flow on a closed 3-manifold $M$, $\gamma$ be a closed orbit of $\Phi_t$.
We assume that
the leaves $W^s(\gamma)$ and $W^u(\gamma)$
of the stable/unstable foliation of $\Phi_t$ containing $\gamma$
are both homeomorphic to an open annuls.
Let $M^*$ be the manifold obtained by blowing-up of $M$ along $\gamma$,
that is, 
for every $x\in \gamma$
replace $x$ to the circle of normal sphere
$\left(T_x M/ T_x \gamma \setminus \{0\}\right) / \mathbb{R}^+ \cong S^1$.
The interior of $M^*$ is homeomorphic to $M\setminus \gamma$
and the boundary $\partial M^*$ is the torus $T_\gamma$.
There is the induced flow $\Phi_t^*$ by blowing-up.
The flow $\Phi_t^*$ has four closed orbits on $T_\gamma$,
two of them are attractive, the other two are repelling.

We take the longitude on $T_\gamma$ as one of the closed orbits of $\Phi_t^*$
and meridian as a normal sphere.
Let $\alpha$ be a simple closed curve on $T_\gamma$
which is transverse to $\Phi_t^*$ and intersects the closed orbits four times.
We can see the slope of $\alpha$ is $1/n$ for some $n\in \mathbb{Z}$.
Note that the curve $\alpha$ is not unique,
and its slope can be chosen to realize any slope $1/n$ for $n\in\mathbb{Z}$.
We take the circle foliation $\mathcal{C}_\alpha$ on $T_\gamma$
which is transverse to $\Phi_t^*$ and parallel to $\alpha$.
Then we define the blowing-down $M^\alpha$ of $M^*$
by identifying every circle leaf of $\mathcal{C}_\alpha$ to a point.
The manifold $M^\alpha$ is the same as the manifold obtained by
$1/n$-Dehn surgery on the closed orbit $\gamma$.
The induced flow $\Phi_t^\alpha$ on $M^\alpha$ becomes an Anosov flow~\cite{Fri83}.

Applying the surgery technique to closed orbits of an Anosov flow,
Fenley showed the following theorem.

\begin{theorem}\label{theoremoffenley}{\rm (Fenley\,\cite[Theorem 2.7]{Fen94})}
Let $\Phi_t$ be an oriented Anosov flow in a closed 3-manifold $M$,
and $\gamma$ be a closed orbit of $\Phi_t$.
If $\Phi_t$ is a suspension Anosov flow,
the induced Anosov flow $\Phi_t^\alpha$ in $M^\alpha$ is $\mathbb{R}$-covered,
where $\Phi_t^\alpha$ and $M^\alpha$ are obtained by $1/n$-Dehn surgery on $\gamma$
as above.
\end{theorem}

We will explain a rough sketch of the proof of Theorem~\ref{theoremoffenley} as follows.
Let $\Phi_t$ be an Anosov flow in a closed 3-manifold $M$.
The orbit space $\mathcal{O}$ of the lifted flow $\tilde{\Phi}_t$ in the universal cover $\tilde{M}$
is homeomorphic to $\mathbb{R}^2$.
Dehn surgery on the closed orbit $\gamma$ in the sense of Fried's surgery
does not change the leaves of the lifted stable foliation $\tilde{\mathcal{F}}^s$,
but changes the manner of connection between the leaves of $\tilde{\mathcal{F}}^u$
which across the leaf of $\tilde{\mathcal{F}}^s$ containing the lift $\tilde{\gamma}$ of the closed orbit $\gamma$.
If $\Phi_t$ is a suspension Anosov flow with an Anosov diffeomorphism $f:T^2\to T^2$,
we can see in the orbit space $\mathcal{O}$ that
the lifted stable foliation $\tilde{\mathcal{F}}^s$ of the resultant Anosov flow $\Phi_t^\alpha$
has the ``positive maximal'' and ``negative maximal'' properties
by using the dilatation constant $\lambda$ of $f$.
The positive and negative maximal properties imply that the leaf space of $\tilde{\mathcal{F}}^s$
has no branching,
then the leaf space of $\mathcal{F}^s$ is homeomorphic to $\mathbb{R}$.
We can see that the leaf space of $\mathcal{F}^u$ is also homeomorphic to $\mathbb{R}$
by reversing the flow direction,
so $\Phi_t^\alpha$ is shown to be $\mathbb{R}$-covered.

Applying Theorem~\ref{theoremoffenley} to a manifold $M$ which contains a GOF-knot $K$
with a hyperbolic monodromy,
we obtain the following.

\begin{prop}\label{maintheorem}
Let $M$ be a closed, orientable 3-manifold and
$K$ be a genus one fibered knot in $M$.
If the monodromy matrix $\phi_\sharp$ of $K$
satisfies $\mathrm{Trace} (\phi_\sharp)>2$,
then the fundamental group $\pi_1 (\Sigma_K(n))$ is left-orderable for any integer $n$,
where $\Sigma_K(n)$ is the closed manifold obtained by Dehn surgery on $K$ with slope $n$. 
\end{prop}

\begin{proof}
Let $K$ be a GOF-knot in $M$ and $\phi\colon T\to T$ be the monodromy of $K$,
and fix a representative monodromy matrix $\phi_\sharp$ of $\phi$.
We take a homeomorphism $\psi\colon T^2 \to T^2$ of a torus
whose representative matrix $\psi_\sharp$ in $SL_2(\mathbb{Z})$
coincides with $\phi_\sharp$.
Since $\mathrm{Trace}(\psi_\sharp)=\mathrm{Trace}(\phi_\sharp)>2$ by the assumption,
there is the suspension Anosov flow $\Phi_t$ on the mapping torus $\widehat{M}_{\psi}$.

Let $\gamma$ be the closed orbit of $\Phi_t$
which corresponds to the origin $0\in T^2=\mathbb{R}^2/\mathbb{Z}^2$.
We take a simple closed curve $\alpha$ with slope $1/n$ for an integer $n$
on the boundary $T_\gamma$ of the blowing-up manifold $\widehat{M}_\psi^*$.
By the blowing-down with the slope of $\alpha$,
we obtain the new manifold $\widehat{M}_\psi^\alpha$,
and the $\mathbb{R}$-covered Anosov flow $\Phi_t^\alpha$ in $\widehat{M}_\psi^\alpha$
by Theorem~\ref{theoremoffenley}.

Let $\Sigma_K(n)$ be the closed manifold which is obtained by Dehn surgery on the GOF-knot $K$
with the slope $n$.
Since the homeomorphism class of $M_\phi$ is determined by the conjugacy class of $\phi_\sharp$ as noted before,
$M_\phi$ is homeomorphic to $\widehat{M}_\psi^*$.
The longitude of $K$ which is the boundary of a fiber $T$
is regarded as the meridian on the boundary $T_\gamma$ of $\widehat{M}_\psi^*$.
Then the meridian-longitude pairs on $\partial M_\phi$ and $\partial \widehat{M}_\psi^*$ are interchanged,
and the slope of $\alpha$ is regarded as the integral slope $n$ on $\partial M_\phi$.
Therefore we can see that $\Sigma_K(n)=M_\phi(n)$ is homeomorphic to $\widehat{M}_\psi^\alpha$
because the manifold obtained by Dehn filling is determined by only the filling slope.

Since $\widehat{M}_\psi^\alpha$ contains the $\mathbb{R}$-covered Anosov flow $\Phi_t^{\alpha}$,
the fundamental group $\pi_1(\widehat{M}_\psi^\alpha)$ is left-orderable
by Proposition~\ref{rcovleftorderable},
so $\pi_1(\Sigma_K(n))$ is also left-orderable.
\end{proof}

\begin{rem}\label{commentshu}
Proposition \ref{maintheorem} has also been proved by Ying Hu using a different method \cite[Corollary 1.8; Remark 1.9]{Hu1}.
\end{rem}

\section{integral surgeries on genus one fibered knots in lens spaces}

In this section,
we apply Proposition~\ref{maintheorem} to GOF-knots in lens spaces.
It is known that the GOF-knots in $S^3$ are only the trefoil and figure-eight knot.
The figure-eight knot $K$ is a GOF-knot with the trace of its monodromy matrix is $3$ (see, \cite[Proposition 5.14]{BZ}).
Then we can apply Proposition~\ref{maintheorem} to the figure-eight knot $K$ and
conclude that for any integer $n\in\mathbb{Z}$ the manifold $\Sigma_K(n)$ has the left-orderable fundamental group.
Note that for the figure-eight knot $K$,
$\Sigma_K(r)$ is not an L-space for all non-trivial slopes $r$~\cite{OS2005Top}.
Then it is conjectured that $\pi_1(\Sigma_K(r))$ is left-orderable for all non-trivial slopes $r$.

In contrast with the case of $S^3$,
there are many GOF-knots in lens spaces.
Morimoto investigated the number of GOF-knots in some lens spaces~\cite{Mo89},
and Baker completely classified the number of GOF-knots in each lens space up to homeomorphisms as follows.

\begin{theorem}\label{theoremofbaker}{\rm (Baker\,\cite[Theorem 4.3]{Ba2014})}
Up to homeomorphisms,
the lens space $L(\alpha', \beta')$ contains exactly
\begin{enumerate}
\item three distinct GOF-knots if and only if $L(\alpha', \beta')\cong L(4,1)$,
\item two distinct GOF-knots if and only if $L(\alpha', \beta')\cong L(\alpha, 1)$ for $\alpha>0$ and $\alpha\neq 4$,
\item one distinct GOF-knot if and only if $L(\alpha', \beta')\cong L(\alpha, \beta)$ either
 \begin{enumerate}
 \item[(3-1)] for $\alpha=0$ or,
 \item[(3-2)] for $0<\beta<\alpha$, where either
  \begin{enumerate}
  \item[(3-2a)] $\alpha=2pq+p+q$ and $\beta=2q+1$ for some integers $p,q>1$, or
  \item[(3-2b)] $\alpha=2pq+p+q+1$ and $\beta=2q+1$ for some integers $p,q>0$, and
  \end{enumerate}
 \end{enumerate}
\item zero GOF-knots otherwise.
\end{enumerate}
\end{theorem}

As seen in Section 2.2,
there is a one-to-one correspondence between
the homeomorphism classes of the pairs $(M, K)$ of a GOF-knot $K$ in a closed 3-manifold $M$
and the homeomorphism classes of the pairs $(L, A)$ of a closed 3-braid $L$ in $S^3$ with its braid axis $A$.
Following the classification in Theorem~\ref{theoremofbaker},
we can interpret each case to the word of 3-braid of $L$.

\begin{lemma}\label{braidrep}
In the classification of Theorem~\ref{theoremofbaker},
each GOF-knot corresponds to the 3-braid as follows:
\begin{enumerate}
\item three distinct GOF-knots in $L(4,1)$ correspond to $A_1:\sigma_1^4\sigma_2$, $A_2:\sigma_1^4\sigma_1^{-1}$ and $A_3:\sigma_1\sigma_2^2\sigma_1\sigma_2^{-1}$
\item two distinct GOF-knots in $L(\alpha, 1)$ correspond to $B_1:\sigma_1^\alpha\sigma_2$ and $B_2:\sigma_1^\alpha\sigma_2^{-1}$ for $\alpha>0$ and $\alpha\neq4$,
\item one GOF-knot
 \begin{enumerate}
 \item[(3-1)] in $L(0,1)$ corresponds to $C:\sigma_2$,
 \item[(3-2)] in $L(\alpha, \beta)$ corresponds to either
  \begin{enumerate}
  \item[(3-2a)] $D_1:\sigma_1^p \sigma_2^2 \sigma_1^q \sigma_2^{-1}$ for $\alpha=2pq+p+q, \beta=2q+1$ $(p,q>1)$, or
  \item[(3-2b)] $D_2:\sigma_1^p \sigma_2^2 \sigma_1^{-q-1} \sigma_2^{-1}$ for $\alpha=2pq+p+q+1, \beta=2q+1$ $(p,q>0)$
  \end{enumerate}
 \end{enumerate}
\end{enumerate}
{\rm(}where we labeled each GOF-knot as $A_1$ to $D_2$ as above{\rm)}.
\end{lemma} 

\begin{proof}
By Hodgson and Rubuistein~\cite{HoRu85},
the lens space $L(\alpha, \beta)$ is the double branched cover of $S^3$ over a link $L$
if and only if $L$ is the two-bridge link $\mathfrak{b}(\alpha, \beta)$.
Baker classifies the equivalence classes of 3-braid representatives of two-bridge knots up to homeomorphism
in \cite[Theorem 4.2]{Ba2014}.
The classification in Theorem~\ref{theoremofbaker} corresponds to this classification of 3-braid representatives.
Along with the proof of \cite[Theorem 4.2]{Ba2014},
we extract these 3-braid representatives of each case as follows.

\noindent \underline{\textit{Case (1)}}
The oriented two-bridge link $\mathfrak{b}(4,1)$ has two closed 3-braid representatives
$\sigma_1^4 \sigma_2$ and $\sigma_1^4 \sigma_2^{-1}$ with braid axes $A_1$ and $A_2$ respectively.
By changing the orientation of one component of $\mathfrak{b}(4,1)$, we obtain $\mathfrak{b}(4,3)$.
$\mathfrak{b}(4,3)$ has the representative $\sigma_1\sigma_2^2\sigma_1\sigma_2^{-1}$ with axis $A_3$.
All these axes are not equivalent each other, then $A_1$, $A_2$, and $A_3$ correspond distinct three GOF-knots
in the lens space $L(4,1)$.

\noindent \underline{\textit{Case (2)}}
By Murasugi~\cite{Mur} (cf. \cite[Proposition 3.1]{Ba2014}),
for $0<\beta<\alpha$ and $\beta$ is odd,
every unoriented two-bridge link of braid index at most 2
is equivalent to $\mathfrak{b}(\alpha,1)$.
The link $\mathfrak{b}(\alpha,1)$ is the type $(2,\alpha)$ torus link.
If $\alpha\neq 4$, they have two inequivalent 3-braid representatives
$\sigma_1^\alpha \sigma_1$ and $\sigma_1^\alpha \sigma_2^{-1}$
(see \cite[Lemma 3.8]{Ba2014}).
If $\alpha=1$ and $\beta=1$, the two-bridge knot $\mathfrak{b}(1,1)$ is unknot.
The unknot has three 3-braid representatives, $\sigma_1\sigma_2$, $\sigma_1^{-1}\sigma_2^{-1}$ and $\sigma_1\sigma_2^{-1}$ (\cite[Lemma 3.4]{Ba2014}, and see also~\cite{BM93}).
However, the two braid axes of $\sigma_1\sigma_2$ and $\sigma_1^{-1}\sigma_2^{-1}$ are equivalent
by an orientation-reversing homeomorphism of $S^3$,
and the braid axis of $\sigma_1\sigma_2^{-1}$ is not equivalent to the other two axes  (\cite[Lemma 3.8]{Ba2014}).
Then $\mathfrak{b}(\alpha,1)$ has two 3-braid representatives
$\sigma_1^\alpha\sigma_2$ and
$\sigma_1^\alpha\sigma_2^{-1}$
whose axes are inequivalent each other
for $\alpha>0$ and $\alpha\neq4$.
Note that the double branched cover of $S^3$ over unknot is also homeomorphic to $S^3$.
The two 3-braid representatives $\sigma_1\sigma_2$ and $\sigma_1\sigma_2^{-1}$ correspond to
the trefoil and figure-eight knot respectively
in the double branched cover. 

\noindent \underline{\textit{Case (3-1)}}
If $\alpha=0$ and $\beta=1$, the two-bridge link $\mathfrak{b}(0,1)$ is the trivial link of two components.
In this case,
$\mathfrak{b}(0,1)$ has just one equivalence class of braid axes representing it as a closed 3-braid
(\cite[Lemma 3.8]{Ba2014}),
and its representative is $\sigma_2$.
Note that the double branched cover over $\mathfrak{b}(0,1)$ is homeomorphic to $S^2\times S^1$
which is recognized as the lens space $L(0,1)$.

\noindent \underline{\textit{Case (3-2a)} and \textit{Case (3-2b)}}
When $0<\beta<\alpha$ and $\beta$ is odd,
the braid index of an oriented two-bridge link $L=\mathfrak{b}(\alpha, \beta)$ is equal to 3
if and only if $\alpha$ and $\beta$ satisfy either condition (3-2a) or (3-2b)
of Lemma~\ref{braidrep}
\cite{Mur} (cf. \cite[Proposition 3.1]{Ba2014}).
These links have just one equivalence class of braid axes representing $L$ as a closed 3-braid
(\cite[Lemma 3.8]{Ba2014}).

If $\alpha=2pq+p+q$ and $\beta=2q+1$ ($p,q>1$),
$\beta/\alpha$ is equal to the continued fraction $[p,2,q]$.
Then the two-bridge knot $\mathfrak{b}(\alpha, \beta)$
has the Conway notation $(p,2,q)$
and its 3-braid representative is equivalent to $\sigma_1^p\sigma_2^2\sigma_1^q\sigma_2^{-1}$.

If $\alpha=2pq+p+q+1$ and $\beta=2q+1$ ($p,q>0$),
$\beta/\alpha$ is equal to the continued fraction $[p,2,-q-1]$.
Then $\mathfrak{b}({\alpha,\beta})$ has the Conway notation $(p,2,-q-1)$
and its representative is  $\sigma_1^p\sigma_2^2\sigma_1^{-q-1}\sigma_2^{-1}$
(see \cite[Figure 3]{Ba2014}).
\end{proof}

By Lemma~\ref{braidrep},
all 3-braid representatives which correspond to the GOF-knots in lens spaces are listed.
By Proposition~\ref{monodromybraid},
the monodromy matrix of a GOF-knot is calculated from its corresponding 3-braid representative.
Applying Proposition~\ref{maintheorem} to the list of monodromy matrices,
we can conclude the following.

\begin{theorem}\label{losurgeryinlensspace}
Let $K$ be a GOF-knot in a lens space $L(\alpha, \beta)$.
If the monodromy matrix $\phi_\sharp$ of $K$ is conjugate to one of the following types {\rm(1)} and {\rm(2)}
in $GL_2(\mathbb{Z})$,
then for any $n\in\mathbb{Z}$,
the fundamental group $\pi_1(\Sigma_K(n))$ of the closed manifold $\Sigma_K(n)$ obtained by $n$-surgery
on $K$
is left-orderable.
\begin{enumerate}
%\item $L(\alpha,\beta)=L(4,1)$ and $\phi_\sharp=\begin{pmatrix} 5 & 4 \\ 1 & 1 \end{pmatrix}$,
\item $L(\alpha,\beta)=L(\alpha,1)$ and $\phi_\sharp=\begin{pmatrix} 1+\alpha & \alpha \\ 1 & 1 \end{pmatrix}$ for $\alpha>0$, or
\item $L(\alpha,\beta)$ satisfies $\alpha=2pq+p+q+1$, $\beta=2q+1$ and \\ $\phi_\sharp=\begin{pmatrix} 2pq+p-q & 2pq+3p-q-1 \\ 2q+1 & 2q+3 \end{pmatrix}$ for integers $p,q>0$.
\end{enumerate}
\end{theorem}

\begin{proof}
The lens space $L(\alpha, \beta)$ containing the GOF-knot $K$
is the double branched cover of $S^3$ over the link $L$
which is the closed 3-braid represented by $\sigma\in B_3=\langle \sigma_1, \sigma_2 \rangle$,
and $K$ is the lift of the braid axis of $\sigma$.
The monodromy matrix $\phi_\sharp$ of $K$ is decomposed to the powers of two matrices $\phi_A$ and $\phi_B$
and these matrices correspond to 3-braid representatives $\sigma_1$ and $\sigma_2^{-1}$
by Proposition~\ref{monodromybraid}.
Then we can obtain the monodromy matrix of $K$ by reading the words in $\sigma$
for each case in Lemma~\ref{braidrep}.
Table 1 is the result of calculations of these matricies and its trace of matrices.
\begin{table}[ht]\label{tableoftrace}
\caption{monodromy matrix of $K$ and its trace}
\begin{tabular}{c|c|c|c}
label & 3-braid rep. & monodoromy matrix $\phi_\sharp$ & $\mathrm{Trace}(\phi_\sharp)$ \\
\hline
$A_1$ & $\sigma_1^4\sigma_2$ & $\begin{pmatrix} -3 & 4 \\ -1 & 1 \end{pmatrix}$ & $-2$ \\
\hline
$A_2$ & $\sigma_1^4\sigma_2^{-1}$ & $\begin{pmatrix} 5 & 4 \\ 1 & 1 \end{pmatrix}$ & $6$ \\
\hline
$A_3$ & $\sigma_1\sigma_2^2\sigma_1\sigma_2^{-1}$ & $\begin{pmatrix} -1 & 0 \\ -3 & -1 \end{pmatrix}$ & $-2$ \\
\hline
$B_1$ & $\sigma_1^\alpha\sigma_2$ & $\begin{pmatrix} 1-\alpha & \alpha \\ -1 & 1 \end{pmatrix}$ & $2-\alpha$ \\
\hline
$B_2$ & $\sigma_1^\alpha\sigma_2^{-1}$ & $\begin{pmatrix} 1+\alpha & \alpha \\ 1 & 1 \end{pmatrix}$ & $2+\alpha$ \\
\hline
$C$ & $\sigma_2$ & $\begin{pmatrix} 1 & 0 \\ -1 & 1 \end{pmatrix}$ & $2$ \\
\hline
$D_1$ & $\sigma_1^p\sigma_2^2\sigma_1^q\sigma_2^{-1}$ & $\begin{pmatrix} -2pq-p+q+1 & -2pq+p+q \\ -2q-1 & -2q+1 \end{pmatrix}$ & $2-q-p(1+2q)$ \\
\hline
$D_2$ & $\sigma_1^p\sigma_2^2\sigma_1^{-q-1}\sigma_2^{-1}$ & $\begin{pmatrix} 2pq+p-q & 2pq+3p-q-1 \\ 2q+1 & 2q+3 \end{pmatrix}$ & $3+p+q+2pq$ \\
\end{tabular}
\end{table}

In Table 1,
the cases of $\mathrm{Trace}(\phi_\sharp)>2$ are $A_2$, $B_2$, and $D_2$,
where $\alpha>0$ in the case $B_2$, and $p,q>0$ in the case $D_2$.
By Proposition~\ref{maintheorem} we can conclude that $\pi_1(\Sigma_K(n))$ is left-orderable in each case.
\end{proof}

\begin{rem}
A codimension one foliation $\mathcal{F}$ in a $3$-manifold is \textit{taut}
if every leaf has a circle intersecting to it which is transverse to the leaves of $\mathcal{F}$ everywhere. 
The closed manifolds obtained in the conclusion of Theorem~\ref{losurgeryinlensspace}
admit transversely orientable stable and unstable foliations of Anosov flows.
These foliations are taut
because a foliation without torus or Klein bottle leaves is taut ~\cite{Goo75},
and all leaves of stable and unstable foliations are not compact, especially not a torus.
In \cite{OS2004GT},
Ozsv\'{a}th and Szab\'{o} showed that
L-space does not contain a transversely orientable taut foliation.
Therefore,
the closed $3$-manifolds obtained in the conclusion of Theorem~\ref{losurgeryinlensspace}
are not L-space, it means that these manifolds satisfy L-space conjecture.
\end{rem}

\begin{rem}\label{remofremaincase}
Among the classification of Theorem \ref{theoremofbaker},
the rest cases of Theorem~\ref{losurgeryinlensspace} are $A_1$, $A_3$, $B_1$, $C$, and $D_1$
which are labeled in Table 1.

In the cases $A_1$, $A_3$, and $C$, all these monodromies $\phi$ are reducible
since $|\mathrm{Trace}(\phi_\sharp)|=2$.
Then these mapping tori $\hat{M}(\phi)$ have a Nil geometry,
so they are Seifert fibered manifolds~\cite[Theorem 4.16]{Sco}.
In the cases $B_1$ with $\alpha=1, 2, 3$, all these monodoromies $\phi$ are periodic
since $\mathrm{Trace}(\phi_\sharp)=-1,0,1$.
Then these $\hat{M}(\phi)$ are also Seifert fibered manifolds.
In all the above cases,
$\Sigma_K(n)$ is a Seifert fibered manifold for any $n\in\mathbb{Z}$~\cite[Proposition 2]{He}.
We can check whether the fundamental group $\pi_1(\Sigma_K(n))$ is left-orderable
by investigating the Seifert invariant of $\Sigma_K(n)$
and using a criterion of Boyer, Rolfsen, and Wiest~\cite{BRW}.

In the cases $B_1$ with $\alpha>4$ and $D_1$,
the once punctured torus bundle $M(\phi)$ is hyperbolic
since $\mathrm{Trace}(\phi_\sharp)<-2$
by Thurston \cite{Th1998a}(see also \cite{Ot} and \cite{Gue}).
By a theorem of Roberts and Shareshian \cite[Corollary 1.5]{RS},
in these cases,
we can see the fundamental group $\pi_1(\Sigma_K(n))$ is not left-orderable
for any integers $n>0$.
But we cannot see the left-orderability of the fundamental groups for $n<0$ in this context.
\end{rem}

By the discussion in Remark~\ref{remofremaincase},
we obtain the following corollary by restricting the cases for hyperbolic
genus one fibered knots in lens spaces.

\begin{coro}\label{hyperboliccase}
Let $K$ be a hyperbolic, genus one fibered knot in lens space $L(\alpha, \beta)$.
The fundamental groups of the closed manifolds obtained by all integral surgeries on K are left-orderable
if and only if the monodromy matrix $\phi_\sharp$ of $K$
is conjugate to one of the types {\rm (1)} and {\rm (2)}
of Theorem~\ref{losurgeryinlensspace} in $GL_2(\mathbb{Z})$.
\end{coro}

\vspace{8pt}
\noindent\textbf{Acknowledgements:}
The authors would like to thank Ying Hu for informing us of the result of \cite{BH} in Remark 1.1, and her result of \cite{Hu1} in Remark 3.4.
The first author is partially supported by JSPS KAKENHI Grant Number 18K03287.
The second author is partially supported by JSPS KAKENHI Grant Number 19K03460.

\end{document}